\documentclass[11pt,a4paper]{amsart}
\usepackage{amsthm}
\usepackage{amsmath,amssymb}
\usepackage{xcolor}
\usepackage{hyperref}
\newtheorem{theorem}{Theorem}
\newtheorem{corollary}[theorem]{Corollary}
\newtheorem{proposition}[theorem]{Proposition}
\newtheorem{lemma}[theorem]{Lemma}
\theoremstyle{remark}
\newtheorem{definition}{Definition}

\numberwithin{theorem}{section}
\newtheorem{remark}[theorem]{Remark}

\def\dx{\,{\rm d}x}
\def\dy{\,{\rm d}y}
\def\dm{\,{\rm d}\mu}
\usepackage[left=1in, right=1in,top=1in,bottom=1in]{geometry}
\allowdisplaybreaks
\def\R{{\mathbb R}}
\def\N{{\mathbb N}}
\numberwithin{equation}{section}
\makeatletter
\@namedef{subjclassname@2020}{%
  \textup{2020} Mathematics Subject Classification}
\makeatother
\begin{document}

\title[Ground state solution for a generalized Choquard  \dots] {Ground state solution for a generalized Choquard Schr\"odinger  equation with vanishing potential in homogeneous fractional Musielak Sobolev spaces}
\author[Shilpa Gupta]
{Shilpa Gupta}
\address{Shilpa Gupta\hfill\break
Department of Mathematics and Statistics\newline
Indian Institute of Technology Kanpur \newline
 Kanpur, 208016, India}
 \email{shilpagupta890@gmail.com; shilpa@iitk.ac.in}

 \author[Gaurav Dwivedi ]
{Gaurav Dwivedi}
\address{Gaurav Dwivedi \hfill\break
Department of Mathematics\newline
 Birla Institute of Technology and Science Pilani \newline
  Pilani Campus, Vidya Vihar \newline
 Pilani, Jhunjhunu \newline
 Rajasthan, India - 333031}
 \email{gaurav.dwivedi@pilani.bits-pilani.ac.in}
\begin{abstract}
This paper aims to establish the existence of  a weak solution for the following problem:
\begin{equation*}
(-\Delta)^{s}_{\mathcal{H}}u(x)  +V(x)h(x,x,|u|)u(x)=\left(\int_{\R^{N}}\dfrac{K(y)F(u(y))}{|x-y|^\lambda}\dy \right) K(x)f(u(x)),  
\end{equation*}
in $\R^{N}$ where   $N\geq 1$, $s\in(0,1), \lambda\in(0,N), \mathcal{H}(x,y,t)=\int_{0}^{|t|} h(x,y,r)r\ dr,$ $ h:\R^{N}\times\R^{N}\times [0,\infty)\rightarrow[0,\infty)$ is a generalized $N$-function and $(-\Delta)^{s}_{\mathcal{H}}$ is a generalized fractional Laplace operator.  The functions $V,K:\R^{N}\rightarrow (0,\infty)$, non-linear function $f:\R\rightarrow \R$ are continuous  and $ F(t)=\int_{0}^{t}f(r)dr.$ 

First, we introduce the homogeneous fractional  Musielak-Sobolev  space and investigate their properties. After that, we pose the given problem in that space.   To establish our existence results, we use variational technique based on the mountain pass theorem. We also prove the existence of a ground state solution by the method of Nehari manifold.
\end{abstract}

\maketitle

\section{Introduction}
\setcounter{equation}{0}
This paper aims to establish the  existence of  a weak solution to the following problem:
\begin{equation}\label{1.1}
(-\Delta)^{s}_{\mathcal{H}}u(x)  +V(x)h(x,x,|u|)u(x)=\left(\int_{\R^{N}}\dfrac{K(y)F(u(y))}{|x-y|^\lambda}\dy \right) K(x)f(u(x)),  
\end{equation}
 in $\R^{N}$ where   $N\geq 1$, $s\in(0,1), \lambda\in(0,N),$  \[\mathcal{H}(x,y,t)=\int_{0}^{|t|} h(x,y,r)r\ dr,\] and $ h:\R^{N}\times\R^{N}\times [0,\infty)\rightarrow[0,\infty)$ is a generalized $N$-function. The functions $V,K:\R^{N}\rightarrow (0,\infty)$, the nonlinear functions $f:\R\rightarrow \R$ are continuous and $ F(t)=\int_{0}^{t}f(r)dr.$ 

The operator $(-\Delta)^{s}_{\mathcal{H}}$ is called the generalized fractional Laplace operator and is defined as:
$$(-\Delta)^{s}_{\mathcal{H}}u(x)=2\lim\limits_{\epsilon\rightarrow 0}\int_{R^{N}\backslash B_{\epsilon}(x)}h\left(x,y, \dfrac{|u(x)-u(y)|}{|x-y|^{s}}\right)\dfrac{u(x)-u(y)}{|x-y|^{s}}\dfrac{\dy}{|x-y|^{N+s}}\cdot$$

The operator $(-\Delta)^{s}_{\mathcal{H}}$ is a generalization of several fractional order operators. More specifically, if we replace $\mathcal{H}$ by $t^p, t^p+t^q$, $b(x)t^p$ and $t^p+a(x)t^q,$  then $(-\Delta)^{s}_{\mathcal{H}}$  reduces to the fractional $p$-Laplacian, fractional $(p,q)$-Laplacian, weighted fractional  Laplace operator and fractional double-phase  operator, respectively. 

The existence results for problems of type \eqref{1.1}  are examined in fractional Orlicz-Sobolev spaces when $\mathcal{H}(x,y,t)$ is independent of $x,y$.   In this context, we quote the work of Bahrouni-Ounaies \cite{sabri}, Bahrouni-Ounaies-Tavares \cite{bahrouni_1},  Missaoui-Ounaies \cite{missaoui}, and Silva-Carvalho-Albuquerque-Bahrouni \cite{silva2}. In the case where $\mathcal{H}(x,y,t)$ depends on all $x,y$ and $t,$ the existence of a solution for problems of the type \eqref{1.1} is studied in fractional Musielak-Sobolev spaces (see Section \ref{sec2}, for the definitions and properties of fractional Musielak-Sobolev spaces.)

The study of Musielak spaces started in the mid-1970s with the work of Musielak \cite{musielak} and Hudzik \cite{hudzik,hudzik1},  where the authors provide the general framework for Musielak spaces in terms of modular function. Recently,  Azroul et al. \cite{musielak1} developed the idea of fractional Musielak-Sobolev  spaces, which are the generalization of fractional Orlicz-Sobolev spaces.
In the same paper, authors also obtained the existence result for the non local problem
 \begin{equation*}
\left\{\begin{array}{ll}
(-\Delta)^{s}_{\mathcal{H}}u(x)+ h\left(x,x,|u|\right)u &= \ f(x,u) \ \ \hbox{in} \ \ \Omega,  \ \ \ \\
\hspace{3.2cm} u  &= \
  0 \ \   \ \hbox{in} \ x\in \R^{N}\backslash \Omega,
\end{array}\right.
\end{equation*}
where $\Omega\subseteq \R^{N}$ is a  bounded and smooth domain,  $N\geq 1$, $0<s<1.$ 
 In the paper \cite{musielak2}, authors proved some embedding and extension  results for fractional Musielak-Sobolev spaces and established the existence result for the following non local problem
\begin{equation*}
\left\{\begin{array}{ll}
(-\Delta)^{s_{1}}_{\mathcal{H}}u(x)+(-\Delta)^{s_{2}}_{\mathcal{H}}u(x)  &= \ f(x,u) \ \ \hbox{in} \ \ \Omega,  \ \ \ \\
\hspace{3.2cm} u  &= \
  0 \ \   \ \ \hbox{in}  \ \ \R^{N}\backslash \Omega,
\end{array}\right.
\end{equation*} 
where $\Omega\subseteq \R^{N}$ is a  bounded and smooth domain,  $N\geq 1$, $0<s_{1}\leq s_{2}<1$. 

To know more about fractional Musielak-Sobolev spaces, kindly refer to the recent work of  Albuquerque-Assis-Carvalho-Salort \cite{Albuquerque} and Gupta \cite{shilpa}.

In this article, we consider the following assumptions on the functions $\mathcal{H}$ and $ h:$
 \begin{itemize}
  \item[$(\mathcal{H}_{1})$]   $ h(x,y,\cdot)\in C^{1}$ in $(0,\infty)$,  $\forall x,y\in\R^N$ and $(th(x,y,t))^{\prime}>0$ $\forall t>0$ and $x,y\in\R^N$.
  \item[$(\mathcal{H}_{2})$]  $h_{1}\leq\frac{ h(x,y,|t|)|t|^{2}}{\mathcal{H}(x,y,|t|)}\leq h_{2}<N$ for all $x,y\in \R^N$ and  $t\neq 0$  for some $1<  h_{1}<h_{2}<h_{1}^{*},$ where $h_{1}^{*}=\frac{Nh_{1}}{N-sh_{1}}\leq h_{2}^{*}=\frac{Nh_{2}}{N-sh_{2}}\cdot$
 \item[$(\mathcal{H}_{3})$]  $\inf\limits_{x,y\in\R^N}\mathcal{H}(x,y,1)=b_{1}$ and $\sup\limits_{x,y\in\R^N}\mathcal{H}(x,y,1)=b_{2}$ for some $b_{1},b_{2}>0$.
 \item[$(\mathcal{H}_{4})$] $\int_{a}^{\infty}\left( \frac{t}{\mathcal{H}(t)}\right)^{\frac{s}{N-s}}dt=\infty$ and $\int_{0}^{b}\left( \frac{t}{\mathcal{H}(t)}\right)^{\frac{s}{N-s}}dt<\infty,$ for some $a,b>0.$
  \item[$(\mathcal{H}_{5})$] $\mathcal{H}(x-z,y-z,t)=\mathcal{H}(x,y,t)$ for all $(x,y)\in\R^N\times\R^N$, $z\in\R^N$ and $t\geq 0$.
 \end{itemize}
 Due to the presence of the Choquard type non linearity, Problem \eqref{1.1} is known as a Choquard equation. One of the main tools to deal with such types of equations is the Hardy-Littlewood-Sobolev \cite{HLS} inequality, which is stated below.
\begin{proposition}\cite{HLS}\label{hardy}
Let $t_{1},t_{2}>1$ and $0<\lambda<N$ with $1/t_{1}+1/t_{2}+\lambda/N=2$, $f\in L^{t_{1}}(\R^{N})$ and $g\in L^{t_{2}}(\R^{N})$. Then there exists a sharp constant $C$ independent of $f$ and $g$ such that 
$$\left| \int_{\R^{N}}\int_{\R^{N}}\dfrac{ f(x)g(y)}{|x-y|^\lambda}\dx \dy\right|\leq C \|f\|_{L^{t_{1}}(\R^{N})}\|g\|_{L^{t_{2}}(\R^{N})}.$$
\end{proposition}

Choquard type of equations have been studied extensively in the literature, we refer to \cite{moroz} for the physical interpretation and survey of such types of equations. For some existence results involving Choquard type equations, we refer to the works of Moroz-Schaftingen \cite{moroz2,moroz3} (Laplace operator), Avenia-Siciliano-Squassina \cite{avenia}, Mukherjee-Sreenadh  \cite{tuhina}, Li-Liu-Li \cite{Liu_Li} (fractional Laplace operator),  Alves-Yang \cite{Alves2} ($p$-Laplacian), Xie-Wang-Zhang \cite {xie} ($(p,q)$-Laplacian), Pucci-Xiang-Zhang \cite{pucci}, Belchior-Bueno-Miyagaki-Pereira \cite{bueno} (fractional $p$-Laplacian) and Zuo-Choudhuri-Repov$\breve{\text{s}}$ \cite{Choudhuri} (variable-order fractional Laplacian). Later, Bahrouni-Ounaies \cite{Bahrouni} discussed the nonlocal Kirchhoff-Choquard problem in fractional Sobolev spaces with variable exponents. Alves-R$\breve{\text{a}}$dulescu-Tavares \cite{choquard} discussed the generalized choquard problem in Orlicz-Sobolev spaces. 

The problem \eqref{1.1} involves the potential term which vanishes at infinity; such a type of equation is widely studied by many researchers. In 2013, Alves-Souto \cite{Alves1} proved the existence result for the equation 
\begin{equation*}
-\Delta u+V(x)u=K(x)f(u) \ \text{in } \R^N,
\end{equation*}
where $N\geq 3$. They assumed that  $V,K:\R^{N}\rightarrow (0,\infty)$ are continuous functions and satisfy the following conditions:
  \begin{itemize}
  \item[$ (K'_{1}) $] $K\in L^{\infty}(\R^{N})$ and if $\{A_{n}\}$ is a sequence of Borel sets such that $\sup\limits_{n}|A_{n}|<\infty$ then
  $$\lim_{s\rightarrow \infty}\int_{A_{n}\cap B_{s}(0)^{c}}K(x)\dx=0 \ \text{uniformly in } n \in\mathbb{N}.$$
 \item[$ (K'_{2}) $] One of the following condition is true:
 \begin{itemize}
   \item[$ (K_{21}) $] $\frac{K}{V}\in L^{\infty}(\R^{N}).$
  \item[$(K_{22})$] $ \dfrac{K(x)}{[V(x)]^{\frac{2^{*}-p}{2^{*}-2}}}\rightarrow 0$ as $|x|\rightarrow \infty$ for some $p\in(2,2^{*})$.
 \end{itemize} 
 \end{itemize} 
 If $V,K$ satisfy $(K'_{1})-(K'_{2})$ then we say $(V,K)\in \mathbb{K}$.
 
Further, Chen-Yuan \cite{chen}, considered the problem:
\begin{equation*}
-\Delta u+V(x)u=\left(\int_{\R^{N}}\dfrac{K(y)F(u(y))}{|x-y|^\lambda}\dy \right) K(x)f(u(x))  \ \hbox{in} \ \R^{N},  
\end{equation*}
 where they assumed that $(V,K)\in \mathbb{K}$ but  the conditions $(K'_{1})$ and $(K_{22})$ are replaced by the conditions $(K_1)$ and $(K_{23})$, respectively. $(K_1)$ and $(K_{23})$ are as follows:
 \begin{itemize}
  \item[$ (K_{1}) $] $K\in L^{\infty}(\R^{N})$ and if $\{A_{n}\}$ is a sequence of Borel sets such that $\sup\limits_{n}|A_{n}|<\infty$ then
  $$\lim_{s\rightarrow \infty}\int_{A_{n}\cap B_{s}(0)^{c}}|K(x)|^{\frac{2N}{2N-\lambda}}\dx=0 \ \text{uniformly in } n \in\mathbb{N}.$$
 \item[$(K_{23})$] $ \dfrac{|K(x)|^{\frac{2N}{2N-\lambda}}}{[V(x)]^{\frac{2^{*}-p}{2^{*}-2}}}\rightarrow 0$ as $|x|\rightarrow \infty$ for some $p\in(2,2^{*})$.
 \end{itemize} 
 
 In what follows, Li-Teng-Wu \cite{li} studied the following fractional Schr$\ddot{\text{o}}$dinger equation:
 \begin{equation*}
 (-\Delta)^{s} u+V(x)u=|u|^{2^{*}_{s}-2}u+\lambda K(x)f(u) \ \text{in } \R^N,
 \end{equation*}
 where $\lambda>0, \ s\in(0,1), \ 2^{*}_{s}=\frac{2N}{N-2s}$ and $(-\Delta)^{s}$ is the fractional Laplace operator of order $s$. They assumed that $(V,K)\in \mathbb{K}$ but in the condition $(K_{22}),$  $ 2^{*}$  is replaced by  $ 2^{*}_{s}.$
 
Luo-Li-Li \cite{luo}, considered the fractional choquard equation:
\begin{equation*}
(-\Delta)^{s} u+V(x)u=\left(\int_{\R^{N}}\dfrac{K(y)F(u(y))}{|x-y|^\lambda}\dy \right) K(x)f(u(x))  \ \hbox{in} \ \R^{N},  
\end{equation*}
in which they assumed that the conditions $(K_{1})$ and 
\begin{itemize}
 \item[$(K_{24})$] $ \dfrac{|K(x)|^{\frac{2N}{2N-\lambda}}}{[V(x)]^{\frac{2_s^{*}-p}{2_s^{*}-2}}}\rightarrow 0$ as $|x|\rightarrow \infty$ for some $p\in(2,2_s^{*})$
 \end{itemize} 
are satisfied.

After that, many researchers studied the nonlinear equations involving vanishing potential with different types of operators and  different conditions on the non linearity, we refer to: Deng-Li-Shuai \cite{deng} ($p$-Laplace operator), Perera-Squassina-Yang \cite{perera} (fractional $p$-Laplacian), Isernia \cite{isernia}  (Fractional $p\&q$-Laplacian), Isernia-Repov$\check{\text{s}}$ \cite{isernia1} (Double-phase operator). Recently, Silva-Souto \cite{silva} developed the existence result for generalized Schr$\ddot{\text{o}}$dinger equation in Orlicz-Sobolev spaces.

Existence results for Choquard type equations with vanishing potential have been obtained by Chen-Yuan \cite{chen}, Alves-Figueiredo-Yang \cite{alves} (for Laplace operator), Albuquerque-Silva-Sousa \cite{Albuquerque} (fractional coupled Choquard-type systems).

In this paper, we assume that $V,K:\R^{N}\rightarrow (0,\infty)$ are continuous functions and satisfy $(K_{1})$.  Moreover, we assume that
\begin{itemize}
\item [$ (K_{2}) $]  $V,K$ satisfy one of the following conditions: 
\begin{itemize}
   \item[$ (K_{2a}) $] $\frac{K}{V}\in L^{\infty}(\R^{N}).$ 
  \item[$ (K_{2b}) $] $\dfrac{|K(x)|^{\frac{2N}{2N-\lambda}}}{L(x)}\rightarrow 0$ as $|x|\rightarrow \infty,$ where $L(x)=\min\limits_{t>0}\left\lbrace V(x)\frac{\mathcal{H}(x,x,t)}{\Psi(x,x,t)}\right\rbrace\cdot$
 \end{itemize} 
 \end{itemize} 
 
Inspired by the mentioned research work, in this article, we study \eqref{1.1} via variational techniques. The main novelties of this paper are to introduce the homogeneous fractional Musielak-Sobolev  spaces and investigate their properties  which are needed to study the problem \eqref{1.1}. Also, provide the characterization of these spaces which is written in the form of Theorem \ref{characterization}. To the best of our knowledge, this is the first paper which proves the existence result for a generalized fractional Laplace operator with the vanishing potential together with Choquard type non linearity.

We assume that $f: \R\rightarrow\R$ is continuous and satisfies the following conditions:
\begin{itemize}
\item[$ (f_{1}) $]   There exists a generalized $N$-function $\Psi:\R^{N}\times\R^{N}\times\R\rightarrow[0,\infty)$  such that $h_{2}<\psi_{1}<\psi_{2}<\psi_{2}l<h_{1}^{*}$,
$\psi_{1}\leq\frac{ \psi(x,y,t)|t|^{2}}{\Psi(x,y,t)}\leq \psi_{2}, \  \forall (x,y)\in \R^{N}\times\R^{N}$, $t\neq 0$  and
$$\displaystyle\lim_{t\rightarrow 0}\frac{f(t)}{\psi(x,x,t)t}=0, \  \forall  \  x \in {\R^{N}},$$
 where $\Psi(x,y,t)=\int_{0}^{|t|} \psi(x,y,r)r\ dr$ and $\frac{2N}{2N-\lambda}=l$. Also, $(t\psi(x,y,t))^{\prime}>0$ $\forall t>0$ and $x,y\in\R^N$.
\item[$ (f_{2}) $] $\displaystyle\lim_{t\rightarrow \infty}\frac{F(t)}{(\mathcal{H}^{*}(x,x,t))^{1/l}}=0, \  \forall  \  x \in {\R^{N}}, $ where   $\mathcal{H}^{*}$ is define in \eqref{criticalexp}.
\item[$ (f_{3}) $] For $i\in\{1,2\}$, $\displaystyle\lim_{t\rightarrow \infty}\frac{f(t)}{(\mathcal{H}^{*}(x,x,t))^{\frac{b-1}{h^{*}_{i}}}}=0, \  \forall  \  x \in {\R^{N}}$  for some $bl\in(h_{2},h_{1}^{*})$.
 \item[$ (f_{4}) $] There exists $\sigma>h_{2}/2$ such that $$0<\sigma F(t)=\sigma\int_{0}^{t}f(s)ds\leq 2tf(t),$$ for all $t>0, \  x \in \R^{N}.$
\end{itemize}

 This article is organized as follows: We discuss the definition and properties of Lebesgue Musielak spaces and  fractional Musielak-Sobolev spaces in Section \ref{sec2}.  The functional setup needed to prove our result is provided in Section \ref{sec3}. We also state our main results in Section \ref{sec3}.  Section \ref{sec4} deals with the proof of Theorem \ref{characterization}. Section \ref{sec5} deals with the proof of Theorem \ref{t1}. Finally, in Section \ref{sec6}, we prove Theorem \ref{t2}.
 
\section{Musielak spaces}\label{sec2}
\setcounter{section}{2} \setcounter{equation}{0}
 Let $\Omega\subseteq \R^{N}$ be any open set. Define \[\mathcal{H}(x,y,t)=\int_{0}^{|t|} h(x,y,s)s\ ds,\] where $ h:\Omega\times\Omega\times [0,\infty)\rightarrow[0,\infty).$

Recall that, $\mathcal{H}(x,y,t):\Omega\times\Omega\times \R\rightarrow[0,\infty)$ is called a generalized $N$-function if it satisfies the following conditions:
\begin{enumerate}
\item $\mathcal{H}$ is  continuous, even and convex function of $t.$
\item $\mathcal{H}(x,y,t)=0$ if and only if $t=0,$ $\forall x,y\in\Omega$.
\item $\lim\limits_{t\rightarrow 0}\frac{\mathcal{H}(x,y,t)}{t}=0$ and $\lim\limits_{t\rightarrow \infty}\frac{\mathcal{H}(x,y,t)}{t}=\infty,$ $\forall x,y\in\Omega$.
\end{enumerate}
For any generalized $N$-function $A:\Omega\times\Omega\times \R\rightarrow[0,\infty),$ we define
 the function $ a_x:\Omega\times [0,\infty)\rightarrow[0,\infty)$ such that $$a_x(x,t)=a(x,x,t) \ \forall (x,t)\in \Omega\times [0,\infty)$$ and  \[A_x(x,t)=\int_{0}^{|t|} a_x(x,s)s\ ds.\]
  
 We say that a generalized $N$-function $\mathcal{H}$, satisfies the weak $\Delta_{2}$-condition if there exist $C>0$ and a non-negative function $k\in L^{1}(\Omega)$ such that 
  $$\mathcal{H}(x,y,2t)\leq C\mathcal{H}(x,y,t)+k(x) \ \ \forall (x\times y\times t)\in\Omega\times\Omega\times[0,\infty).$$ If $k=0;$ then $\mathcal{H}$ is said to satisfies the  $\Delta_{2}$-condition. Throughout this paper, we assume that $\mathcal{H}$ is a generalized $N$-function which is locally integrable, i.e., for any $t>0$ and for every compact set $K\subseteq\Omega$, we have
  \begin{equation}\label{loc_int}
      \int_{K\times K}\mathcal{H}(x,y,t)\dx\dy< \infty \text{ and } \int_{K}\mathcal{H}(x,x,t)\dx<\infty.\end{equation}

Next, we define the complementary function $\widehat{\mathcal{H}}$ corresponding to generalized $N$-function $\mathcal{H}$ as 
$$\widehat{\mathcal{H}}(x,y,t)=\int_{0}^{|t|}\widehat{ h}(x,y,s)s\ ds,$$ where  $\widehat{ h}$ is  defined as $\widehat{ h}(x,y,t)=\sup\{s: h(x,y,s)s\leq t\} \ \forall(x,y,t)\in{\Omega}\times{\Omega}\times [0,\infty).$  

It can be observed, the condition $(\mathcal{H}_{2})$ implies that the generalized $N$-function $\mathcal{H}$ and its complementary function $\widehat{\mathcal{H}}$ satisfy the $\Delta_{2}$-condition.

Moreover, the function $\mathcal{H}$ and its complementary function $\widehat{\mathcal{H}}$ satisfy the following  Young's inequality \cite[Proposition 2.1]{liu2015}: 
  $$s_{1}s_{2}\leq\mathcal{H}(x,y,s_{1})+\widehat{\mathcal{H}}(x,y,s_{2}) \ \forall x,y\in\Omega, s_{1},s_{2}>0.$$

  
The Lebesgue-Musielak space $L^{\mathcal{H}_{x}}(\Omega)$ is defined as:
 \begin{align*}
L^{\mathcal{H}_{x}}(\Omega)=\left\lbrace u:\Omega\rightarrow\R \ \text{is measurable} \right. & \left| \right.\int_{\Omega}\mathcal{H}_x\left( x,\tau|u|\right)\dx<\infty, \\
& \left. \text{for some} \  \tau>0\right\rbrace \cdot
\end{align*}
$L^{\mathcal{H}_{x}}(\Omega)$ is a normed space \cite{musielak}  with the Luxemburg norm
 $$\| u\|_{L^{\mathcal{H}_{x}}(\Omega)}=\inf\left\lbrace \tau>0\left| \ \int_{\Omega}\mathcal{H}_x\left( x,\dfrac{|u|}{\tau}\right) \dx\leq 1\right\rbrace \right. \cdot$$ 
 \begin{theorem}\cite[Remark B.1 and Theorem B.3]{Youssfi}
Let $\mathcal{H}$ be any generalized $N$-function which satisfies $(\mathcal{H}_{2})$.  Then the space $L^{\mathcal{H}_{x}}(\Omega)$ is a separable and reflexive Banach space.
 \end{theorem}

\begin{theorem}\cite[Theorem 2.2]{Youssfi}\label{density}
Let $\mathcal{H}$ be any generalized $N$-function. If $\mathcal{H}$ is locally integrable, then $C_c^{\infty}(\Omega)$ is dense in $L^{\mathcal{H}_{x}}(\Omega)$.
 \end{theorem}
 
We denote by $B_c(\Omega)$ the set of compactly supported bounded functions in $\Omega$.
For any $h\in\R^N$ and for any function $u$, we define the translation operator $\tau$ of $u$ as\begin{equation*}
\tau_h(u)(x)=\left\{\begin{array}{ll}
u(x+h) \ \hspace{0.5cm} \hbox{if} \ x\in\Omega \hbox{ and } x+h\in\Omega \ \ \ \\
 0  \ \hspace{1.6cm} \hbox{otherwise}.
\end{array}\right.
\end{equation*}
If the function $u$ has a compact support, $\tau_h(u)$ is well defined provided that $h<dist(\text{supp } u, \partial \Omega)$.
\begin{theorem}\cite[Theorem 2.1]{Youssfi}\label{translation}
Let $\mathcal{H}$ be any generalized $N$-function satisfying \eqref{loc_int}. Let $u\in B_c(\Omega)$. Then, for every $\varepsilon>0$, there exists $\eta>0$(depending on $\varepsilon$) such that  
$$\|\tau_h(u)-u\|< \varepsilon \text{ for all } h\in\R^N \text{ with } |h|<\eta.$$
\end{theorem}
\begin{proposition}\label{prop5}\cite{adams}
Let $\mathcal{H}$ and $\widehat{\mathcal{H}}$ be complementary $N$-functions. Then, for any $u\in L^{\mathcal{H}_{x}}(\Omega)$ and $v\in L^{\widehat{\mathcal{H}}_{x}}(\Omega)$, we have
$$\left| \int_{\Omega}uv\ \dx\right|\leq 2 \|u\|_{L^{\mathcal{H}_{x}}(\Omega)}\|v\|_{L^{\widehat{\mathcal{H}}_{x}}(\Omega)}.$$
\end{proposition}
\begin{lemma}\label{dual}\cite[Lemma B.5]{Youssfi}
Let $v\in L^{\widehat{\mathcal{H}}_{x}}(\Omega)$. Then 
\begin{equation}\label{gc}
G_{v}(u)=\int_{\Omega}u(x)v(x)\dx
\end{equation} is a bounded linear functional on $L^{\mathcal{H}_x}(\Omega)$, i.e., $G_{v}\in (L^{\mathcal{H}_x}(\Omega))^{*}$. Also, every bounded linear functional in $L^{\mathcal{H}_x}(\Omega)$ is of the form \eqref{gc} for some $v\in L^{\widehat{\mathcal{H}}_{x}}(\Omega)$. Moreover, $(L^{\mathcal{H}_x}(\Omega))^{*}$ is isomorphic to $L^{\widehat{\mathcal{H}}_{x}}(\Omega).$ 
 \end{lemma}
 \begin{remark}\label{dual_rem}
By Lemma \ref{dual}, the norm $\|\cdot\|_{L^{\widehat{\mathcal{H}}_x}(\Omega)}$ is equivalent to the norm $\|\cdot\|_{(L^{\mathcal{H}_x}(\Omega))^{*}}$, i.e., 
$$\|v\|_{L^{\widehat{\mathcal{H}}_x}(\Omega)}\leq\|G_{v}\|_{(L^{\mathcal{H}_x}(\Omega))^{*}}=\sup_{\|u\|_{L^{\mathcal{H}_x}(\Omega)}\leq 1}\left\lbrace \left| \int_{\Omega}u(x)v(x)\dx\right| \right\rbrace\leq 2\|v\|_{L^{\widehat{\mathcal{H}}_x}(\Omega)}\
\cdot  $$ 
\end{remark} 

For a given generalized $N$-function $\mathcal{H}$ and $s\in(0,1)$, fractional  Musielak-Sobolev  space is  denoted by $W^{s,\mathcal{H}}(\Omega)$ and is defined as
\begin{align*}
W^{s,\mathcal{H}}(\Omega)=\left\lbrace u\in L^{\mathcal{H}_{x}}(\Omega):\right.&\int_{\Omega}\int_{\Omega}\mathcal{H}\left(x,y, \dfrac{\tau|u(x)-u(y)|}{|x-y|^{s}}\right)\dfrac{\dx \dy}{|x-y|^{N}} <\infty,\\
&\left. \text{for some }   \tau>0\right\rbrace. 
\end{align*}
$W^{s,\mathcal{H}}(\Omega)$ is a normed space  with the  norm
$$\| u\|=\| u\|_{L^{\mathcal{H}_{x}}(\Omega)}+[u]_{s,\mathcal{H}},$$ 
where
$$[u]_{s,\mathcal{H}}=\inf\left\lbrace \tau>0\left| \ \int_{\Omega}\int_{\Omega}\mathcal{H}\left(x,y, \dfrac{|u(x)-u(y)|}{\tau|x-y|^{s}}\right)\dfrac{\dx \dy}{|x-y|^{N}}\leq 1\right\rbrace \right.\cdot$$

We define the Lebesgue-Musielak space $L^{\mathcal{H}}(\dm)$ as:
\begin{align*}
L^{\mathcal{H}}(\dm)=\left\lbrace u:\right.\Omega\times\Omega\rightarrow\R \ \text{is measurable}&\left| \right. \int_{\Omega}\int_{\Omega}\mathcal{H}\left(x,y,\tau|u(x,y)|\right) \dm<\infty, \\
&\left. \text{for some } \  \tau>0\right\rbrace,
\end{align*}
 where $\dm=\dfrac{\dx  \dy}{|x-y|^N}$ is a Borel measure on the set $\Omega\times\Omega$. 
 \begin{remark}\label{rem2}
 $[u]_{s,\mathcal{H}}$ is finite if and only if $\dfrac{(u(x)-u(y))}{|x-y|^{s}}\in L^{\mathcal{H}}(\dm)$ and $[u]_{s,\mathcal{H}}=\left\|\dfrac{u(x)-u(y)}{|x-y|^{s}} \right\|_{L^{\mathcal{H}}(\dm)}\cdot$
\end{remark} 
 \begin{theorem}\cite{musielak1}
 Let $\mathcal{H}$ be any generalized $N$-function which satisfies $(\mathcal{H}_{2})$. Then $W^{s,\mathcal{H}}(\Omega)$ is  a separable and reflexive Banach space.
\end{theorem}

The space $W_{0}^{s,\mathcal{H}}(\Omega)$ is defined as 
$$ W_{0}^{s,\mathcal{H}}(\Omega) =\{u\in W^{s,\mathcal{H}}(\R^{N}): \ u=0 \ \text{ a.e. in} \  \R^{N}\backslash \Omega  \}.$$
Next, we state the generalized Poincar$\acute{{\text{e}}}$'s inequality:
\begin{theorem}{\cite{musielak1}}
Let $\Omega$ be a bounded open subset of $ \ \R^{N}$ and $0<s<1.$ Then there exists a positive constant $c>0$ such that
$$\| u\|_{L^{\mathcal{H}_{x}}(\Omega)}\leq c [u]_{s,\mathcal{H}}, \ \ \forall \ u\in  W_{0}^{s,\mathcal{H}}(\Omega).$$
\end{theorem}

This implies that, $[\cdot]_{s,\mathcal{H}}$ is the norm on $W_{0}^{s,\mathcal{H}}(\Omega),$ which is equivalent to the norm $\| \cdot\|.$ 

For a given generalized $N$-function $\mathcal{H}:\R^{N}\times\R^{N}\times \R\rightarrow[0,\infty)$, we define the Sobolev conjugate function $\mathcal{H}^{*}:\R^{N}\times \R\rightarrow[0,\infty)$  as:
\begin{equation}\label{criticalexp}\mathcal{H}^{*}(x,t)=\mathcal{H}_x(x,G^{-1}(t)), \forall \ t\geq 0,
\end{equation}
where
$$G(x,t)=\left( \int_{0}^{t}\left( \frac{r}{\mathcal{H}_x(x,r)}\right)^{\frac{s}{N-S}}dr\right)^{\frac{N-s}{N}}, \forall \ t\geq 0.$$
One can verify that $\mathcal{H}^{*}$ is a generalized $N$-function.
\begin{theorem}\label{optm}
Let $s\in(0,1)$ and  $\mathcal{H}$ be any generalized $N$-function satisfying $(
\mathcal{H}_{4})$.
Then the embedding $W^{s,\mathcal{H}}(\R^{N})\hookrightarrow L^{\mathcal{H}^{*}}(\R^{N})$ is continuous. Also, we have $\|u\|_{L^{\mathcal{H}^{*}}(\R^{N})}\leq c [u]_{s,\mathcal{H}}< \infty$ for all $u\in W^{s,\mathcal{H}}(\R^{N})$ for some $c>0.$ Moreover, in this embedding the space $L^{\mathcal{H}^{*}}(\R^{N})$ is optimal among all the Musielak spaces.
\end{theorem}
\begin{proof}
The proof is similar to the proof of \cite[Theorem 6.1]{optimal}. We omit the details.
\end{proof}
\begin{proposition}\label{prop7}
Let $\mathcal{H}$ be any generalized $N$-function satisfying $(\mathcal{H}_{2})-(\mathcal{H}_{3})$. Assume that $u\in L^{\mathcal{H}_x}(\R^{N}).$ Then, we have
\begin{enumerate}
\item[$(1)$] $\min\left\lbrace \rho^{h_{1}},\rho^{h_{2}}\right\rbrace\mathcal{H}_x(x, t)\leq \mathcal{H}_x(x,\rho t)\leq\max\left\lbrace \rho^{h_{1}},\rho^{h_{2}}\right\rbrace\mathcal{H}_x(x, t), \ \forall \rho,t>0$.
 \item[(2)] \begin{align*}
\min\left\lbrace \|u\|^{h_{1}}_{L^{\mathcal{H}_x}(\R^{N})},\|u\|^{h_{2}}_{L^{\mathcal{H}_x}(\R^{N})}\right\rbrace&\leq \int_{\R^{N}}\mathcal{H}_x(x,|u|)\dx\\
&\leq\max\left\lbrace \|u\|^{h_{1}}_{L^{\mathcal{H}_x}(\R^{N})},\|u\|^{h_{2}}_{L^{\mathcal{H}_x}(\R^{N})}\right\rbrace\cdot
\end{align*}
\end{enumerate}
\end{proposition}
\begin{proof}
Proof of $(1)$ is similar to the proof of  \cite[Lemma 2.1]{fukagai2006}. By $(\mathcal{H}_{3})$  and $(1)$, we have
\begin{equation}\label{minmax}
b_{1}\min\left\lbrace \rho^{h_{1}},\rho^{h_{2}}\right\rbrace\leq \mathcal{H}_x(x,\rho )\leq b_{2} \max\left\lbrace \rho^{h_{1}},\rho^{h_{2}}\right\rbrace  \ \forall \rho>0.
\end{equation}
Hence,  \eqref{minmax} and the definition of norm implies $(2).$
\end{proof}

Let $\mathcal{H}$ be any generalized $N$-function satisfying $(\mathcal{H}_{2})-(\mathcal{H}_{3})$. Then we define  weighted Lebesgue-Musielak space $L_{V}^{\mathcal{H}_x}(\R^{N})$  as:
\begin{align*}
L_{V}^{\mathcal{H}_x}(\R^{N})=\left\lbrace u:\R^{N}\rightarrow\R \ \text{is measurable}\right.\left| \right.& \int_{\R^{N}}V(x)\mathcal{H}_x\left( x,\tau|u|\right) \dx<\infty, \\
& \left. \text{for some } \tau>0\right\rbrace\cdot
\end{align*}
$L_{V}^{\mathcal{H}_x}(\R^{N})$ is a normed space \cite[Section 5]{Opic}  with the Luxemburg norm
$$\| u\|_{V,\mathcal{H}}=\inf\left\lbrace \tau>0\left| \ \int_{\R^{N}}V(x)\mathcal{H}_x\left(
 x,\tau|u|\right) \dx\leq 1\right\rbrace \right. \cdot$$ 
\begin{corollary}\label{col1}
Let $\mathcal{H}$ be any generalized $N$-function satisfying $(\mathcal{H}_{2})-(\mathcal{H}_{3})$. Assume that $u\in L_{V}^{\mathcal{H}_x}(\R^{N})$. Then, we have
   $$\min\left\lbrace \|u\|^{h_{1}}_{V,\mathcal{H}},\|u\|^{h_{2}}_{V,\mathcal{H}}\right\rbrace\leq \int_{\R^{N}}V(x)\mathcal{H}_x(x,|u|)\dx\leq\max\left\lbrace \|u\|^{h_{1}}_{V,\mathcal{H}},\|u\|^{h_{2}}_{V,\mathcal{H}}\right\rbrace\cdot$$
\end{corollary}
\begin{proposition}\label{prop9}\cite[Lemma 4.3]{missaoui}
Let $\mathcal{H}$ be any generalized $N$-function satisfying $(\mathcal{H}_{2})$. Assume that $u\in \mathcal{H}^{*}(\R^{N})$ and $\rho,t\geq 0.$ Then, we have
\begin{enumerate}
\item   $\min\left\lbrace \rho^{h^{*}_{1}},\rho^{h^{*}_{2}}\right\rbrace\mathcal{H}^{*}(x, t)\leq \mathcal{H}^{*}(x,\rho t)\leq\max\left\lbrace \rho^{h^{*}_{1}},\rho^{h^{*}_{2}}\right\rbrace\mathcal{H}^{*}(x, t),$
\item  \begin{align*}
\min\left\lbrace \|u\|^{h^{*}_{1}}_{L^{\mathcal{H^{*}}}(\R^{N})},\|u\|^{h^{*}_{2}}_{L^{\mathcal{H}^{*}}(\R^{N})}\right\rbrace&\leq \int_{\R^{N}}\mathcal{H}^{*}(x,|u|)\dx\\
&\leq\max\left\lbrace \|u\|^{h^{*}_{1}}_{L^{\mathcal{H}^{*}}(\R^{N})},\|u\|^{h^{*}_{2}}_{L^{\mathcal{H}^{*}}(\R^{N})}\right\rbrace,\end{align*}
\end{enumerate}
where, $h_{1}^{*}=\frac{Nh_{1}}{N-sh_{1}}$ and $h_{2}^{*}=\frac{Nh_{2}}{N-sh_{2}}\cdot$
\end{proposition}
\subsection{Homogeneous fractional  Musielak-Sobolev  space} The fractional Musielak-Sobolev spaces are not sufficient to study problem \eqref{1.1}, as $\inf V(x)$ can be zero. In this section, we introduce the suitable space to study Problem \eqref{1.1} which we call homogeneous fractional Musielak-Sobolev space, and investigate its properties. 

One can verify that the space $C_{c}^{\infty}(\R^{N})$ is normed space with the norm $[\cdot]_{s,\mathcal{H}}$. However, the normed space $(C_{c}^{\infty}(\R^{N}),[\cdot]_{s,\mathcal{H}})$ is not complete. Further, we define the  completion $D^{s,\mathcal{H}}(\R^{N})$ of $(C_{c}^{\infty}(\R^{N}),[\cdot]_{s,\mathcal{H}})$  in the standard way. 
More precisely, 
\begin{align*}
D^{s,\mathcal{H}}(\R^{N})=&\left\lbrace  [u_{n}]: \{u_{n}\}\subseteq C_{c}^{\infty}(\R^{N}) \text{ is a Cauchy sequence}\right. \\
& \left. \text{under the norm } [\cdot]_{s,\mathcal{H}} \right\rbrace,
\end{align*}
where $[u_{n}]$ is the equivalence class of the Cauchy sequence $\{u_{n}\}$ with the equivalence relation $'\sim_{s,\mathcal{H}}',$ which is defined as $\{u_{n}\}\sim_{s,\mathcal{H}}\{v_{n}\}$ iff $\lim\limits_{n\rightarrow\infty}[u_{n}-v_{n}]_{s,\mathcal{H}}=0.$
$D^{s,\mathcal{H}}(\R^{N})$ is the Banach space with the norm $\|[u_{n}]\|_{D^{s,\mathcal{H}}(\R^{N})}=\lim\limits_{n\rightarrow\infty}[u_{n}]_{s,\mathcal{H}}$.

Next, we define the characterization of  the normed space $(D^{s,\mathcal{H}}(\R^{N}),\|\cdot\|_{D^{s,\mathcal{H}}(\R^{N})})$.
Consider the space
$$\mathring{W}^{s,\mathcal{H}}(\R^{N})=\left\lbrace u\in L^{\mathcal{H}^{*}}(\R^{N}):  [u]_{s,\mathcal{H}}<\infty\right\rbrace.$$ 

$\mathring{W}^{s,\mathcal{H}}(\R^{N})$ is a normed space with the norm  $\|u\|_{\mathring{W}^{s,\mathcal{H}}(\R^{N})}=[u]_{s,\mathcal{H}}+\|u\|_{L^{\mathcal{H}^{*}}(\R^{N})}$ which is equivalent to the norm $[u]_{s,\mathcal{H}}$ (by Theorem \ref{optm}).
\begin{theorem}\label{characterization}
Let $\mathcal{H}$ be a generalized $N$-function and $s\in(0,1).$ Then $C_{c}^{\infty}(\R^{N})$ is the dense subspace of $\mathring{W}^{s,\mathcal{H}}(\R^{N})$. Moreover,
there exists  a linear isomorphism between  $\mathring{W}^{s,\mathcal{H}}(\R^{N})$ and $D^{s,\mathcal{H}}(\R^{N})$. In other words, the space $D^{s,\mathcal{H}}(\R^{N})$ can be identified as $\mathring{W}^{s,\mathcal{H}}(\R^{N})$ and $\|\cdot\|_{D^{s,\mathcal{H}}(\R^{N})}=\|\cdot\|_{\mathring{W}^{s,\mathcal{H}}(\R^{N})}=[\cdot]_{s,\mathcal{H}}$.
\end{theorem}

We provide a proof of Theorem \ref{characterization} in Section \ref{sec4}.
Due to the presence of potential term $V$ in the Problem \eqref{1.1}, we consider the following weighted space:
$$ W=\left\lbrace u\in D^{s,\mathcal{H}}(\R^{N}): \int_{\R^{N}}V(x)\mathcal{H}_{x}\left( x,|u|\right) \dx<\infty\right\rbrace$$ which is a normed space with the norm 
$$\|u\|_{W}=\|u\|_{D^{s,\mathcal{H}}(\R^{N})}+\|u\|_{V,\mathcal{H}}\cdot$$ 
For the sake of simplicity, we denote $\|\cdot\|_{W}$ as $\|\cdot\|$.

Next, we have the following lemma from the definition of the space $W$.
\begin{lemma}\label{compact}
The space $W$ is compactly embedded in $L_{loc}^{\mathcal{H}}(\R^{N})$. Also, $W$ is continuously embedded in $L^{\mathcal{H}^{*}}(\R^{N})$.
\end{lemma}
\begin{proof}
Let $K$ be any compact subset of $\R^{N}$ then $\|u\|_{ W^{s,\mathcal{H}}(K)}\leq \|u\|_{D^{s,\mathcal{H}}(\R^{N})}$ which implies embedding $ D^{s,\mathcal{H}}(\R^{N})\hookrightarrow W^{s,\mathcal{H}}(K)$ is continuous. It is proved in \cite[Theorem 2.2]{musielak2} that the embedding $W^{s,\mathcal{H}}(\omega)\hookrightarrow\hookrightarrow L^{\mathcal{H}}(\omega)$ is compact. Also, by definition of $W$, we have the embedding $W\hookrightarrow D^{s,\mathcal{H}}(\R^{N})$ is continuous. By the composition of continuous and compact embedding, we have $W$ is compactly embedded in $L_{loc}^{\mathcal{H}}(\R^{N})$. Moreover, by the definition of  $W$ and norm on $\mathring{W}^{s,\mathcal{H}}(\R^{N})$, we get $W\hookrightarrow L^{\mathcal{H}^{*}}(\R^{N})$ continuously.
\end{proof}

Next, we will state some results which are used to prove our main result.
Define the function $m:\mathring{W}^{s,\mathcal{H}}(\R^{N})\rightarrow\R$ as
$$m(u)=\int_{\R^{N}}\int_{\R^{N}}\mathcal{H}\left(x,y, \dfrac{|u(x)-u(y)|}{|x-y|^{s}}\right)\dfrac{\dx \dy}{|x-y|^{N}}\cdot$$
\begin{proposition}\cite{musielak1}\label{prop8}
For all $u\in \mathring{W}^{s,\mathcal{H}}(\R^{N})$ we have
\begin{enumerate}
\item  If  $[u]_{s,\mathcal{H}}> 1$ then $[u]_{s,\mathcal{H}}^{h_{1}}\leq m(u)\leq [u]_{s,\mathcal{H}}^{h_{2}}$.
\item  If  $[u]_{s,\mathcal{H}}< 1$ then $[u]_{s,\mathcal{H}}^{h_{2}}\leq m(u)\leq [u]_{s,\mathcal{H}}^{h_{1}}$.
\end{enumerate}
In particular, $m(u)=1$ iff $[u]_{s,\mathcal{H}}=1$. Moreover, if $\{u_{n}\}\subset \mathring{W}^{s,\mathcal{H}}(\R^{N})$ then $\|u_{n}\|\rightarrow 0$ iff $m(u_{n})\rightarrow 0.$
\end{proposition}
\begin{theorem}\label{embed}
 Suppose that $K$ be any continuous function satisfies $(K_{2})$ then  the space $W$ is continuously embedded in $L^{\Psi}_{Q}(\R^{N})$, where $Q(x)=|K(x)|^l$, $\frac{2N}{2N-\lambda}=l$ and $\psi:\R^{N}\times\R^{N}\times[0,\infty)\rightarrow[0,\infty)$, $\Psi(x,y,t)=\int_{0}^{|t|} \psi(x,y,r)r\ dr$ is a generalized $N$-function  such that 
 $$\psi_{1}\leq\frac{ \psi(x,y,t)|t|^{2}}{\Psi(x,y,t)}\leq \psi_{2},  \forall (x,y)\in \R^{N}\times\R^{N} \text{ and } t\neq 0$$   for some $\psi_{1},\psi_{2} \in (h_{2},h_{1}^{*})\cdot$
\end{theorem}
\begin{proof}
By condition $(K_{2b})$ there exists $r>0$ such that
\begin{equation}\label{e1}
Q(x)\Psi(x,x,t)\leq V(x)\mathcal{H}(x,x,t) \text{ for all } t>0 \text{  and } |x|\geq r.
\end{equation}
By Proposition \ref{prop7} and \ref{prop9} and using the fact $h_{2}<\psi_{1}<\psi_{2}<h_{1}^{*}$, we have $\lim\limits_{t\rightarrow 0}\frac{\Psi(x,x,t)}{\mathcal{H}(x,x,t)}=0$ and $\lim\limits_{t\rightarrow \infty}\frac{\Psi(x,x,t)}{\mathcal{H}^{*}(x,x,t)}=0, \forall x\in \R^{N}$, i.e., \begin{equation*}
	\Psi(x,x,t)\leq c_{1}\mathcal{H}(x,x,t)+c_{2}\mathcal{H}^{*}(x,x,t) \text{ for all } t>0  \text{ and } x\in \R^{N},
\end{equation*}	
for some $c_{1},c_{2}>0$.
Therefore, by condition $(K_{1b})$, we have
\begin{align*}
	Q(x)\Psi(x,x,t)\leq &c_{1}\left\| \frac{K}{V}\right\|_{L^{\infty}(B(0,r))}|K(x)|^{\frac{\lambda}{2N-\lambda}} V(x)\mathcal{H}(x,x,t)\\
    &+c_{2}|K(x)|^l\mathcal{H}^{*}(x,x,t)
\end{align*}
for all  $t>0$ and $|x|\leq r,$
which implies
\begin{equation}\label{e2}
	Q(x)\Psi(x,x,t)\leq c_{3}V(x)\mathcal{H}(x,x,t)+c_{4}\mathcal{H}^{*}(x,x,t), 
\end{equation}
 for all  $t>0$ and $|x|\leq r,$
where $c_{3}=c_{1}\left\| \frac{K}{V}\right\|_{L^{\infty}(B(0,r))}\max\limits_{x\in B(0,r)} |K(x)|^{\frac{\lambda}{2N-\lambda}}$ and $c_{4}=c_{2}\max\limits_{x\in B(0,r)} |K(x)|^{l}$.
By \eqref{e1} and \eqref{e2}, we have
\begin{equation}\label{e3}
	Q(x)\Psi(x,x,t)\leq c_{5}V(x)\mathcal{H}(x,x,t)+ c_{4}\mathcal{H}^{*}(x,x,t)\text{ for all } t>0 \text{  and }  x\in \R^{N},
\end{equation}
where $c_{5}=\max\{1,c_{3}\}$. Let $u\in W$ be any arbitrary element then by Lemma \ref{compact} and \eqref{e3}, we have
\begin{align*}
	\int_{\R^{N}}	Q(x)\Psi\left( x,x,\frac{|u|}{\|u\|}\right)\dx &\leq c_{5}\int_{\R^{N}}	V(x)\mathcal{H}\left( x,x,\frac{|u|}{\|u\|_{V,\mathcal{H}}}\right)\dx \\
    &+c_{4}\int_{\R^{N}}	\mathcal{H}^{*}\left( x,x,\frac{|u|}{\|u\|_{L^{\mathcal{H}^{*}}(\R^{N})}}\right)\dx\leq c_{6},
	\end{align*}
	for some $c_{6}>0$, hence $u\in L^{\Psi}_{Q}(\R^{N})$ and by the definition of  the norm we have $$\|u\|_{L^{\Psi}_{Q}(\R^{N})}\leq c_{7}\|u\|$$ for some $c_{7}>0$, which implies $W$ is continuously embedded in $L^{\Psi}_{Q}(\R^{N})$.
\end{proof}

As we have discussed, the Hardy-Littlewood-Sobolev inequality is the primary tool for dealing with the Choquard type non linearity in the context of variational methods. So far, we do not have the Hardy-Littlewood-Sobolev inequality for Lebesgue Musielak spaces. Taking advantage of the condition $(\mathcal{H}_{2})$ and using Proposition \ref{hardy}, we prove and use the following result in order to control the Choquard term.
\begin{proposition}\label{l00}
Let $\mathcal{H}$ be any generalized $N$-function satisfying $(\mathcal{H}_{2})-(\mathcal{H}_{3})$, $(K_{2})$ and $(f_{1})-(f_{2})$. For any $u\in W,$ we have
$K(x)F(u(x))\in L^{l}(\R^{N})$. Moreover, for all $\epsilon>0$ there exists $c_{\epsilon}>0$ such that 
\begin{align*}
\left| \int_{\R^{N}}\int_{\R^{N}}\right.  & \left. \dfrac{K(x) K(y) F(u(x))F(u(y))}{|x-y|^\lambda}\dx \dy\right|\\
&\leq  C_1\max\left\lbrace{\epsilon}^{2}\|u\|^{2\psi_{1}}+c_{\epsilon}^{2}\|u\|^{2h_{1}^{*}/l}, {\epsilon}^{2}\|u\|^{2\psi_{2}}+c_{\epsilon}^{2}\|u\|^{2h_{2}^{*}/l}\right\rbrace\\
&\leq C_2 (\|u\|^{2\psi_{1}}+\|u\|^{2h_{1}^{*}/l}+\|u\|^{2\psi_{2}}+\|u\|^{2h_{2}^{*}/l})
\end{align*}
for some $C_1,C_2>0.$
\end{proposition}
\begin{proof}
Let $u\in W$. It follows from  $ (f_{1})- (f_{2})$ that, for all $\epsilon>0$ there exists $c_{\epsilon}>0$ such that
\begin{equation*}
|F(t)|\leq  \epsilon{\Psi_x}(x,t)+c_{\epsilon}(\mathcal{H}^{*}(x,t))^{1/l}, \forall (x,t)\in\R^{N}\times\R.
\end{equation*}
 By Proposition \ref{prop7}, we have 
\begin{align*}
&\int_{\R^{N}}  |K(x)F(u(x))|^{l}\dx\leq  2^{l-1}\int_{\R^{N}}(\epsilon({\Psi_x}(x,u(x)))^{l}+c_{\epsilon}^{l}{\mathcal{H}^{*}}(x,u(x)))\dx\\
& \leq  c_{1}{\epsilon}^{l}\max\left\lbrace \|u\|_{L^{\psi_{1}l}(\R^{N})}^{\psi_{1}l},\|u\|_{L^{\psi_{2}l}(\R^{N})}^{\psi_{2}l}\right\rbrace+c_{2}c_{\epsilon}^{l}\max\left\lbrace \|u\|_{L^{\mathcal{H}^{*}}(\R^{N})}^{h_{1}^{*}},\|u\|_{L^{\mathcal{H}^{*}}(\R^{N})}^{h_{2}^{*}}\right\rbrace
\end{align*}
for some $c_{1},c_{2}>0.$
Further, by Theorem \ref{embed} and Lemma \ref{compact}, one gets
\begin{align}\label{l1.1}
\int_{\R^{N}}  |K(x)F(u(x))|^{l}\dx\leq &c_{3}{\epsilon}^{l} \max\left\lbrace \|u\|^{\psi_{1}l},\|u\|^{\psi_{2}l}\right\rbrace\notag\\
&+c_{4}c_{\epsilon}^{l}\max\left\lbrace \|u\|^{h_{1}^{*}},\|u\|^{h_{2}^{*}}\right\rbrace
<\infty,
\end{align}
which implies, $K(x)F(u(x))\in L^{l}(\R^{N})$, for some $c_{3},c_{4}>0.$

By Proposition \ref{hardy} and \eqref{l1.1}, we get
\begin{align*}
\left| \int_{\R^{N}}\right.&\int_{\R^{N}} \left. \dfrac{K(x) K(y) F(u(x))F(u(y))}{|x-y|^\lambda}\dx \dy\right|\\
&\leq  \left( c_{3}{\epsilon}^{l} \max\left\lbrace \|u\|^{\psi_{1}l},\|u\|^{\psi_{2}l}\right\rbrace+c_{4}c_{\epsilon}^{l}\max\left\lbrace \|u\|^{h_{1}^{*}},\|u\|^{h_{2}^{*}}\right\rbrace\right) ^{2/l}\\
&\leq  C_1\max\left\lbrace{\epsilon}^{2}\|u\|^{2\psi_{1}}+c_{\epsilon}^{2}\|u\|^{2h_{1}^{*}/l},{\epsilon}^{2}\|u\|^{2\psi_{2}}+c_{\epsilon}^{2}\|u\|^{2h_{2}^{*}/l}\right\rbrace\\
&\leq C_2 (\|u\|^{2\psi_{1}}+\|u\|^{2h_{1}^{*}/l}+\|u\|^{2\psi_{2}}+\|u\|^{2h_{2}^{*}/l})
\end{align*}
for some $C_1,C_2>0.$
\end{proof}

\section{Functional Setting}\label{sec3}
\setcounter{section}{3} \setcounter{equation}{0}
First, we define a weak solution to \eqref{1.1} and the corresponding energy functional. 
\begin{definition}\normalfont
We say that  $u\in  W $ is a weak solution of  \eqref{1.1} if the following holds:
\begin{equation*}
\begin{split}
&\int_{\R^{N}}\int_{\R^{N}}  h\left(x,y, \dfrac{|u(x)-u(y)|}{|x-y|^{s}}\right)\dfrac{(u(x)-u(y))(v(x)-v(y))}{|x-y|^{N+2s}}\dx \dy\\
&+ \int_{\R^{N}}V(x)h_x(x,|u|)u v\dx=\int_{\R^{N}}\int_{\R^{N}}\dfrac{K(x) K(y) F(u(x))f(u(y))v(y)}{|x-y|^\lambda}\dx \dy ,  
\end{split}
\end{equation*}
for all $ v\in W.$
\end{definition}

Thus, the energy functional $I:W\rightarrow \R$ corresponding to \eqref{1.1} is given by 
\begin{align*}
I(u)=\int_{\R^{N}}\int_{\R^{N}}\mathcal{H}&\left( \dfrac{|u(x)-u(y)|}{|x-y|^{s}}\right)\dfrac{\dx \dy}{|x-y|^{N}}+ \int_{\R^{N}}V(x)\mathcal{H}_{x}(x,|u|)\dx \\
&-\frac{1}{2}\int_{\R^{N}}\int_{\R^{N}}\dfrac{K(x) K(y) F(u(x))F(u(y))}{|x-y|^\lambda}\dx \dy.
\end{align*}
It can be seen that $I$ is well defined by Proposition \ref{l00}, $C^{1}$ \cite[Lemma 3.2]{choquard} and the derivative of $I$ at any point $u\in W$ is given by
\begin{align*}
I'&(u)(v)=\int_{\R^{N}}\int_{\R^{N}} h\left(x,y,\dfrac{|u(x)-u(y)|}{|x-y|^{s}}\right)\dfrac{(u(x)-u(y))(v(x)-v(y))}{|x-y|^{N+2s}}\dx \dy\\
&+ \int_{\R^{N}}V(x)h_{x}(x,|u|)u v\dx-\int_{\R^{N}}\int_{\R^{N}}\dfrac{K(x) K(y) F(u(x))f(u(y))v(y)}{|x-y|^\lambda}\dx \dy  ,
\end{align*}
for all $ v\in W.$
Moreover, the critical points of $I$ are the weak solutions to  \eqref{1.1}. 

Let $J:W\rightarrow \R$ be such that
$$J(u)=\int_{\R^{N}}\int_{\R^{N}}\mathcal{H}\left(x,y, \dfrac{|u(x)-u(y)|}{|x-y|^{s}}\right)\dfrac{\dx\dy}{|x-y|^{N}}+ \int_{\R^{N}}V(x)\mathcal{H}_{x}(x,|u|)\dx.$$
\begin{remark}
The functional $J$ is convex, since $\mathcal{H}$ is convex. Consequently, $J$ is weakly lower semicontinuous, i.e., if $\{u_{n}\}\rightharpoonup u$ in $W$ then 
$J(u)\leq\liminf\limits_{n\rightarrow\infty}J(u_{n})$.
\end{remark}

Next, proceeding as in \cite[Theorem 3.14]{Albuquerque} together with \cite[Lemma 2.8]{Bahrouni_Missaoui}, we obtain the following result:
\begin{lemma}\label{ps}
Let $\mathcal{H}$ be a generalized $N$-function and $s\in(0,1).$ Assume that the  sequence $\{u_{n}\}$ converges weakly to $u$ in $W$ and
$$\limsup_{n\rightarrow\infty}\langle J'(u_{n}),u_{n}-u\rangle\leq 0.$$
Then $\{u_{n}\}$ converges strongly to $u$ in $W$.
\end{lemma}

The main existence result of this paper is as follows:
\begin{theorem}\label{t1}
Suppose  that the conditions $(f_{1})-(f_{4})$, $({K}_{1})-({K}_{2})$ and $(\mathcal{H}_{1})-(\mathcal{H}_{5})$  are satisfied. Then   the  Problem \eqref{1.1} has a nontrivial  weak solution.
\end{theorem}

To prove the existence of ground state solution, we need the following additional assumption on $f$:
 \begin{itemize}
 \item[$(GS)$]  The map  $t\mapsto\frac{f(t)}{t|t|^{\frac{h_{2}}{2}-2}}$ is strictly increasing for $t>0$.
 \end{itemize}
\begin{theorem}\label{t2} If $(f_{1})-(f_{4})$, $(GS)$, $({K}_{1})-({K}_{2})$ and $(\mathcal{H}_{1})-(\mathcal{H}_{5})$
 are satisfied, then the solution obtained through Theorem \ref{t1} is a ground state solution.
  \end{theorem} 
\section{Proof of Theorem \ref{characterization}}\label{sec4} 
\setcounter{section}{4} \setcounter{equation}{0}
The proof of Theorem \ref{characterization} is inspired by the work of Baalal-Berghout-Ouali \cite{Baalal}, which employs the fundamental technique of convolution combined with the use of a cut-off function.
 \begin{proof}
We present the proof of the theorem in three steps:

\textbf{Step $ 1 $:} In this step, we will prove that $C_{c}^{\infty}(\R^{N})$ is dense in $\mathring{W}^{s,\mathcal{H}}(\R^{N})$, i.e., for any $u\in\mathring{W}^{s,\mathcal{H}}(\R^{N})$ there exists a sequence in $(C_{c}^{\infty}(\R^{N}),[\cdot]_{s,\mathcal{H}})$ which converges to $u$ in $\mathring{W}^{s,\mathcal{H}}(\R^{N})$.

Let $\rho\in C_{c}^{\infty}(\R^{N})$ be the standard mollifier with support inside $B_{1}(0)$. Define, $\rho_\epsilon(x)=\epsilon^{-n}\rho\left(\frac{x}{\epsilon} \right) $. It can be seen that $\rho_\epsilon(x)\in C_{c}^{\infty}(\R^{N})$, $\int_{\R^{N}}\rho_\epsilon(x) \dx=1$ and support of $\rho_\epsilon$ belongs to $B_{\epsilon}(0)$.

Let $u\in\mathring{W}^{s,\mathcal{H}}(\R^{N})$ be any arbitrary element. Then $u_\epsilon=\rho_\epsilon\ast u\in C^{\infty}(\R^{N})$. Next, we claim that $[u_\epsilon-u]_{s,\mathcal{H}}\rightarrow 0$ as $\epsilon\rightarrow 0$. 

By using  Proposition \ref{prop5}, Remarks \ref{dual_rem}, \ref{rem2} and the properties of mollifiers, we have
\begin{align*}
[u_\epsilon-u]&_{s,\mathcal{H}}=\left\|\dfrac{(u_\epsilon(x)-u(x))-(u_\epsilon(y)-u(y))}{|x-y|^{s}} \right\|_{L^{\mathcal{H}}(\dm)}\\
&  \leq\sup_{\|v\|_{L^{\widehat{\mathcal{H}}}(\dm)}\leq 1}\left\lbrace \left|\int_{\R^{N}}\int_{\R^{N}} \dfrac{(u_\epsilon(x)-u(x))-(u_\epsilon(y)-u(y))}{|x-y|^{s}}v(x,y)\dm\right| \right\rbrace\\
&\leq \sup_{\|v\|_{L^{\widehat{\mathcal{H}}}(\dm)}\leq 1}2\|v\|_{L^{\widehat{\mathcal{H}}}(\dm)} \\
&\left\lbrace \int_{|\xi|<1}\rho(\xi)d\xi \left\|  \dfrac{(u(x-\epsilon\xi)-u(y-\epsilon\xi))-(u(x)-u(y))}{|x-y|^{s}}\right\|_{L^{\mathcal{H}}(\dm)}  \right\rbrace\\
&= 2  \int_{|\xi|<1}\rho(\xi) \left\|  \dfrac{(u(x-\epsilon\xi)-u(y-\epsilon\xi))-(u(x)-u(y))}{|x-y|^{s}}\right\|_{L^{\mathcal{H}}(\dm)}  d\xi.
\end{align*}
As we know that $w(x,y)=\dfrac{|u(x)-u(y)|}{|x-y|^{s}}\in L^{\mathcal{H}}(\dm)$ and $C_{c}^{\infty}(\dm)$ is dense in $L^{\mathcal{H}}(\dm)$ (by Theorem \ref{density}), therefore, for a given $\sigma>0$ there exists $k(x,y)\in C_{c}^{\infty}(\dm)$ such that 
\begin{equation}\label{ep1}
\|w-k\|_{L^{\mathcal{H}}(\dm)}\leq \frac{\sigma}{3}.
\end{equation}
Furthermore, by the triangle inequality of the norm, we obtain 
\begin{align}\label{ep2}
\notag\|&w(x-\epsilon\xi,  y-\epsilon\xi)-w(x, y)\|_{L^{\mathcal{H}}(\dm)}\notag\\
&\leq \|w(x-\epsilon\xi, y-\epsilon\xi)-k(x-\epsilon\xi, y-\epsilon\xi)\|_{L^{\mathcal{H}}(\dm)}\notag\\
+&
\|k(x-\epsilon\xi, y-\epsilon\xi)-k(x, y)\|_{L^{\mathcal{H}}(\dm)}+\|k(x, y)-w(x, y)\|_{L^{\mathcal{H}}(\dm)}. 
\end{align}
By using $(\mathcal{H}_{5})$ and \eqref{ep1}, we obtain
\begin{equation}\label{ep3}
\|w(x-\epsilon\xi, y-\epsilon\xi)-k(x-\epsilon\xi, y-\epsilon\xi)\|_{L^{\mathcal{H}}(\dm)}<\frac{\sigma}{3}.   
\end{equation}
As $k\in C_{c}^{\infty}(\dm)$, using the Theorem \ref{translation} (for sufficeintly small $\epsilon$), we have 
\begin{equation}\label{ep4}
\|k(x-\epsilon\xi, y-\epsilon\xi)-k(x,y)\|_{L^{\mathcal{H}}(\dm)}<\frac{\sigma}{3}.   
\end{equation}
Finally, substituting the estimates from \eqref{ep1}, \eqref{ep3}, and \eqref{ep4} into the inequality \eqref{ep2}, we get
$$\|w(x-\epsilon\xi, y-\epsilon\xi)-w(x, y)\|_{L^{\mathcal{H}}(\dm)}<\sigma,$$
for sufficiently small $\epsilon$.
Therefore, we get $[u_\epsilon-u]_{s,\mathcal{H}}\leq\sigma$, for sufficiently small $\epsilon$. As $ \sigma $ was arbitrary, we get $[u_\epsilon-u]_{s,\mathcal{H}}\rightarrow 0$ as $\epsilon\rightarrow 0$.

\textbf{Step $ 2 $:} Let $\{u_{n}\}\subseteq (C_{c}^{\infty}(\R^{N}),[\cdot]_{s,\mathcal{H}})$ be a Cauchy sequence. 

Claim: There exists $u\in \mathring{W}^{s,\mathcal{H}}(\R^{N})$ such that $u_n\rightarrow u$ in $\mathring{W}^{s,\mathcal{H}}(\R^{N})$. 

By Theorem \ref{optm}, we have
$$\|u_n\|_{L^{\mathcal{H}^{*}}(\R^{N})}\leq c [u_n]_{s,\mathcal{H}}< \infty, \ \forall n\in \N.$$ 
Hence, $\{u_{n}\}\subseteq L^{\mathcal{H}^{*}}(\R^{N})$ and  $\{u_{n}\}$ a Cauchy sequence in $L^{\mathcal{H}^{*}}(\R^{N})$. As we know that $L^{\mathcal{H}^{*}}(\R^{N})$ is a Banach space; thus, there exists $u\in L^{\mathcal{H}^{*}}(\R^{N})$ such that $u_n\rightarrow u$ in $L^{\mathcal{H}^{*}}(\R^{N})$. This implies that, $u_n(x)\rightarrow u(x)$ a.e. in $\R^{N}$. By the continuity of $\mathcal{H}$,  we have
$\mathcal{H}\left(x,y, \dfrac{|u_n(x)-u_n(y)|}{|x-y|^{s}}\right)\rightarrow \mathcal{H}\left(x,y, \dfrac{|u(x)-u(y)|}{|x-y|^{s}}\right)$ a.e. in $\R^{N}$.

Thanks to the Fatou lemma,
\begin{align*}
\int_{\R^{N}}\int_{\R^{N}}&\mathcal{H}\left(x,y, \dfrac{|u(x)-u(y)|}{|x-y|^{s}}\right)\dm \\
&\leq \liminf_{n\rightarrow \infty} \int_{\R^{N}}\int_{\R^{N}}\mathcal{H}\left(x,y, \dfrac{|u_n(x)-u_n(y)|}{|x-y|^{s}}\right) \dm <\infty,
\end{align*}
which implies $u\in \mathring{W}^{s,\mathcal{H}}(\R^{N})$.

Next, we will prove that  $[u_n-u]_{s,\mathcal{H}}\rightarrow 0$ as $n\rightarrow \infty$.

As $\{u_{n}\}\subseteq (C_{c}^{\infty}(\R^{N}),[\cdot]_{s,\mathcal{H}})$, we have
\begin{align*}\int_{\R^{N}}\int_{\R^{N}}&\mathcal{H}\left(x,y, \dfrac{|u_n(x)-u_n(y)|}{|x-y|^{s}}\right)\dm\\
&=\int_{\R^{N}}\int_{\R^{N}}\mathcal{H}\left(x,y, \dfrac{|u_n(x)-u_n(y)|}{|x-y|^{s}}\right)\dfrac{\dx\dy}{|x-y|^{N}}<\infty \end{align*}
for each $n\in \N.$ Thus $ \dfrac{|u_n(x)-u_n(y)|}{|x-y|^{s}}\in L^{\mathcal{H}}(\dm)$. 

Let $$z_n(x,y)=\dfrac{|u_n(x)-u_n(y)|}{|x-y|^{s}}\cdot$$

It can also be seen that $\left\lbrace z_n(x,y)\right\rbrace $  is a Cauchy sequence in $ L^{\mathcal{H}}(\dm)$. As $ L^{\mathcal{H}}(\dm)$ is a Banach space, there exists $z(x,y)\in L^{\mathcal{H}}(\dm)$ such that $z_n\rightarrow z$ in $L^{\mathcal{H}}(\dm)$. Further, by uniqueness of the limit, we have $z(x,y)=\dfrac{|u(x)-u(y)|}{|x-y|^{s}}\cdot$ Hence, $[u_n-u]_{s,\mathcal{H}}\rightarrow 0$ as $n\rightarrow \infty,$ which proves our claim.

\textbf{Step $ 3 $:} Let $[u_{n}]\in D^{s,\mathcal{H}}(\R^{N})$, i.e. $[u_{n}]$ is an equivalence class of the Cauchy sequence $\{u_{n}\}\subseteq (C_{c}^{\infty}(\R^{N}),[\cdot]_{s,\mathcal{H}})$. By	Step $ 2 $, there exists $u\in\mathring{W}^{s,\mathcal{H}}(\R^{N})$ such that $[u_n-u]_{s,\mathcal{H}}\rightarrow 0$ as $n\rightarrow \infty$.
 Define a function, $\Bbbk:D^{s,\mathcal{H}}(\R^{N})\rightarrow \mathring{W}^{s,\mathcal{H}}(\R^{N})$ such that
 $\Bbbk([u_{n}])=u$. It can be see that $\Bbbk$ is well defined one-one, onto and isometry by Step $ 1 $ and Step $ 2 $, which completes the proof the theorem.
  \end{proof}

\section{Proof of Theorem \ref{t1}}  \label{sec5}  
\setcounter{section}{5} \setcounter{equation}{0}
To prove our main result, we first establish a series of lemmas.
\begin{lemma}\label{l1}
There exist positive real numbers $\alpha$ and $\rho$ such that 
$$I(u)\geq \alpha, \ \ \forall u \in   W :\|u\|=\rho.$$
\end{lemma}
\begin{proof}
By using the Corollary \ref{col1}  and  Proposition \ref{prop8}, we have
\begin{equation*} 
\begin{split}
I(u)&=\int_{\R^{N}}\int_{\R^{N}}\mathcal{H}\left(x,y,\dfrac{|u(x)-u(y)|}{|x-y|^{s}}\right)\dfrac{\dx \dy}{|x-y|^{N}}+ \int_{\R^{N}}V(x)\mathcal{H}_{x}(x,|u|)\dx \\
&-\frac{1}{2}\int_{\R^{N}}\int_{\R^{N}}\dfrac{K(x) K(y) F(u(x))F(u(y))}{|x-y|^\lambda}\dx \dy\\
 & \geq     \min\left\lbrace [u]_{s,\mathcal{H}}^{h_{1}},[u]_{s,\mathcal{H}}^{h_{2}}\right\rbrace+  \min\left\lbrace \|u\|_{V,\mathcal{H}}^{h_{1}},\|u\|_{V,\mathcal{H}}^{h_{2}}\right\rbrace\\
 &-\frac{1}{2}\int_{\R^{N}}\int_{\R^{N}}\dfrac{K(x) K(y) F(u(x))F(u(y))}{|x-y|^\lambda}\dx \dy\cdot
 \end{split}
\end{equation*}
If $\|u\|<1,$  Proposition \ref{l00} implies
\begin{equation*} 
\begin{split}
I(u)&\geq \|u\|^{h_{2}}- (C_{1}{\epsilon}^{2}\|u\|^{2\psi_{1}}+C_{1}c_{\epsilon}^{2}\|u\|^{2h_{1}^{*}/l})\\
&\geq\|u\|^{h_{2}}\left( 1-\frac{C_{1}{\epsilon}^{2}}{\|u\|^{h_{2}-2\psi_{1}}}\right)-C_{1}c_{\epsilon}^{2}\|u\|^{2h_{1}^{*}/l}. 
 \end{split}
\end{equation*}

We conclude the result by choosing $\rho$ and $\epsilon$ sufficiently small and using the fact that $(2h_{1}^{*}/l)>h_{2}$.
\end{proof}	
\begin{lemma}\label{l2}
 There exist $\nu_{0}\in   W $ and   $\beta>0$  such that 
$$I(\nu_{0})< 0 \ \ \hbox{and} \ \|\nu_{0}\|>\beta.$$ 
\end{lemma}
\begin{proof}
By $(f_{4})$,  there exist $m_{1},m_{2}>0$ such that
\begin{equation*}
F(s)\geq m_{1}s^{\sigma}-m_{2}, \ \ \forall \ s\in  [0,\infty).
\end{equation*}
Let $u\in W\backslash\{0\}$ and $u\geq 0$ with compact support $K\subseteq \R^{N}$. For $t>1$, by Corollary \ref{col1}  and  Proposition \ref{prop8}, we have
\begin{equation*} 
\begin{split}
I(tu)&=\int_{\R^{N}}\int_{\R^{N}}\mathcal{H}\left( x,y,\dfrac{|tu(x)-tu(y)|}{|x-y|^{s}}\right)\dfrac{\dx\dy}{|x-y|^{N}}+ \int_{\R^{N}}V(x)\mathcal{H}_{x}(x,|tu|)\dx \\
&-\frac{1}{2}\int_{K}\int_{K}\dfrac{K(x) K(y) F(tu(x))F(tu(y))}{|x-y|^\lambda}\dx \dy\\
 & \leq t^{h_{2}}\left(  \max\left\lbrace [u]_{s,\mathcal{H}}^{h_{1}},[u]_{s,\mathcal{H}}^{h_{2}}\right\rbrace+  \max\left\lbrace \|u\|_{V,\mathcal{H}}^{h_{1}},\|u\|_{V,\mathcal{H}}^{h_{2}}\right\rbrace\right) \\
 &-\frac{1}{2}\int_{K}\int_{K}\dfrac{K(x) K(y) (m_{1}t^{\sigma}(u(x))^{\sigma}-m_{2})(m_{1}t^{\sigma}(u(y))^{\sigma}-m_{2})}{|x-y|^\lambda}\dx \dy
 \end{split}
\end{equation*}
this implies that $I(tu)\rightarrow -\infty$ as $t\rightarrow\infty$, since $2\sigma>h_{2}$. Now, by setting $\nu_{0}=tu$ for sufficiently large $t$, we get the desired result. 
\end{proof}

By  Lemmas \ref{l1} and \ref{l2}, the geometric conditions of the mountain pass theorem   are satisfied for the functional $I$.  Hence, by the version of the mountain pass theorem without (PS) condition, there exists a sequence  $\{u_{n}\}\subseteq  W $ such that  $I(u_{n})\rightarrow c_{M}$ and $I'(u_{n})\rightarrow 0$ as  $n\rightarrow \infty$, where
$$0<c_{M}=\inf_{\gamma\in \varGamma}\max_{t\in [0,1]}I(\gamma(t))>0,$$
and
$$\varGamma=\{\gamma\in C([0,1], W ):\gamma(0)=0, \ \gamma(1)<0\}\cdot$$
\begin{lemma}\label{l3}
  The $(PS)_{c_{M}}$ sequence is bounded in $W$. Moreover, there exists $u\in  W $ such that, up to a subsequence, we have
  $u_{n}\rightharpoonup u$ weakly in $ W $.
\end{lemma}
 \begin{proof}
 Since  $\{u_{n}\}$ is a $(PS)_{c_{M}}$ sequence of $I$, we have
$I(u_{n})\rightarrow c_{M}$ and $I'(u_{n})\rightarrow 0$ as  $n\rightarrow \infty$, i.e.,
\begin{align}\label{p11}
\int_{\R^{N}}\int_{\R^{N}}\mathcal{H}&\left( \dfrac{|u_{n}(x)-u_{n}(y)|}{|x-y|^{s}}\right)\dfrac{\dx\dy}{|x-y|^{N}}+ \int_{\R^{N}}V(x)\mathcal{H}_{x}(x,|u_{n}|)\dx \\
&-\frac{1}{2}\int_{\R^{N}}\int_{\R^{N}}\dfrac{K(x) K(y) F(u_{n}(x))F(u_{n}(y))}{|x-y|^\lambda}\dx \dy =c_{M}+\delta_{n},
\notag
\end{align}
where $\delta_{n}\rightarrow 0$ as $n\rightarrow\infty$
and 
\begin{equation}\label{p22}
\begin{split}
&\left| \int_{\R^{N}}\int_{\R^{N}}\right. h\left(x,y,\dfrac{|u_{n}(x)-u_{n}(y)|}{|x-y|^{s}}\right)\dfrac{(u_{n}(x)-u_{n}(y))(v(x)-v(y))}{|x-y|^{N+2s}}\dx\dy \\
&  + \int_{\R^{N}}V(x)h_{x}(x,|u_{n}|)u_{n} v\dx\\
&\left.-\int_{\R^{N}}\int_{\R^{N}}\dfrac{K(x) K(y) F(u_{n}(x))f(u_{n}(y))v(y)}{|x-y|^\lambda}\dx \dy \right|\leq \varepsilon_{n}\|v\|,  
\end{split}
\end{equation}
$\forall v\in  W ,$ where $\varepsilon_{n}\rightarrow 0$ as $n\rightarrow\infty$.
On taking $v=u_n,$ by \eqref{p11}, \eqref{p22} and using $(f_{4})$, we obtain
\begin{equation*}
\begin{split}
\left( \int_{\R^{N}}\int_{\R^{N}}\right.&\mathcal{H}\left(x,y, \dfrac{|u_{n}(x)-u_{n}(y)|}{|x-y|^{s}}\right)  \dfrac{\dx\dy}{|x-y|^{N}}\\
&\left.-\frac{1}{\sigma}\int_{\R^{N}}\int_{\R^{N}} h\left(x,y,\dfrac{|u_{n}(x)-u_{n}(y)|}{|x-y|^{s}}\right)\dfrac{(u_{n}(x)-u_{n}(y))^{2}}{|x-y|^{N+2s}}\dx \dy\right) \\
&+ \int_{\R^{N}}(V(x)\mathcal{H}_{x}(x,|u_n|)-\frac{1}{\sigma}V(x)h_{x}(x,|u_n|)u_n^2)\dx
\leq c_{5}(1+\|u_{n}\|),
\end{split}
\end{equation*}
for some $c_{5}>0$. It follows from $(\mathcal{H}_{2})$ that
\begin{equation*}
\begin{split}
\left(1-\frac{h_{2}}{\sigma}\right) \int_{\R^{N}}&\int_{\R^{N}}\mathcal{H}\left(x,y, \dfrac{|u_{n}(x)-u_{n}(y)|}{|x-y|^{s}}\right)\dfrac{\dx\dy}{|x-y|^{N}}\\
&+\left(1-\frac{h_{2}}{\sigma}\right) \int_{\R^{N}}V(x)\mathcal{H}_{x}(x,|u_{n}|)\dx 
\leq c_{5}(1+\|u_{n}\|).
\end{split}
\end{equation*}
If $\|u_{n}\|> 1,$ by Corollary \ref{col1}  and  Proposition \ref{prop8}, we have
\begin{align*}
\left(1-\frac{h_{2}}{\sigma}\right)([u_{n}]_{s,\mathcal{H}}^{h_{1}}+\|u_{n}\|_{V,\mathcal{H}}^{h_{1}})&\leq c_{5}(1+\|u_{n}\|)\\
\left(1-\frac{h_{2}}{\sigma}\right)\|u_{n}\|^{h_{1}}&\leq c_{5}(1+\|u_{n}\|).
\end{align*}
Consequently, 
$\|u_{n}\|\leq c_{6}$
for some $c_{6}>0$.
Thus $\{u_{n}\}$ is bounded in $ W $. As $ W $ is a reflexive Banach space, there exists $u\in  W $ such that  up to a subsequence, we have $u_{n}\rightharpoonup u$ weakly in $ W $.
\end{proof}


 \begin{lemma}\label{l5}
  Let $\{u_{n}\}$ is bounded in $W$ such that  $u_{n}\rightharpoonup u$ weakly in $W$. Then
 \begin{equation*}
\lim_{n\rightarrow\infty}\int_{\R^{N}}|K(x)f(u_{n}(x))(u_{n}(x)-u(x))|^{l}\dx = 0.
\end{equation*}
  \end{lemma}
\begin{proof}
Let $\{u_{n}\}$ is bounded in $W$ such that  $u_{n}\rightharpoonup u$ weakly in $W$. By Lemma  \ref{compact}, we have $u_{n}(x)\rightarrow u(x)$ a.e. $x\in\R^{N}.$ 

Define $Q(x)=|K(x)|^{l}$. It follows from  $(f_{1})$ and $(f_{3})$ that, for all $\epsilon>0$ there exist $t_{0},t_{1},c_{\epsilon}>0$ such that
\begin{equation}\label{3.5}
f(t)\leq  \epsilon \left( \psi_x(x,t)t+(\mathcal{H}^{*}(x,t))^{\frac{b-1}{h^{*}_{i}}}\right)+c_{\epsilon} (\mathcal{H}^{*}(x,t))^{\frac{b-1}{h^{*}_{i}}}\chi_{[t_{0},t_{1}]}(t),
\end{equation}
for all $(x,t)\in\R^{N}\times\R$.
Further, by  \eqref{3.5}, we have
\begin{align*}
K(x)f(u_{n}(x))&\leq  \epsilon K(x)\left( \psi_x(x,u_{n}(x))u_{n}(x) +(\mathcal{H}^{*}(x,u_{n}(x)))^{\frac{b-1}{h^{*}_{i}}}\right)\\
&+c_{\epsilon} K(x)(\mathcal{H}^{*}(x,u_{n}(x)))^{\frac{b-1}{h^{*}_{i}}}\chi_{[t_{0},t_{1}]}(u_{n}(x)), \forall (x,t)\in\R^{N}\times\R.
\end{align*}
Consider,
\begin{equation}\label{ps2}
\begin{split}
&\int_{\R^{N}} |K(x)f(u_{n}(x))(u_{n}(x)-u(x))|^{l}\dx\leq  2^{l-1}\epsilon^{l}* \\
&\int_{\R^{N}}Q(x)\left| \left\lbrace \psi_x(x,u_{n}(x))u_{n}(x)+  (\mathcal{H}^{*}(x,u_{n}(x)))^{\frac{b-1}{h^{*}_{i}}}\right\rbrace (u_{n}(x)-u(x))\right|^{l}\dx\\
&+2^{l-1}c_{\epsilon}^{l}\int_{\R^{N}}Q(x)\left| (\mathcal{H}^{*}(x,u_{n}(x)))^{\frac{b-1}{h^{*}_{i}}}\chi_{[t_{0},t_{1}]}(u_{n}(x))(u_{n}(x)-u(x))\right| ^{l}\dx.
\end{split}
\end{equation}
Now, define the set $A_n=\{x\in \R^N: |u_n(x)|\geq t_0 \}$. Thus,  $(\mathcal{H}_{3})$ and  definition of $ \mathcal{H}^{*}$ implies
$$c_{7}|A_n|\leq \int_{A_n}\mathcal{H}^{*}(x,t_0)\dx \leq \int_{A_n}\mathcal{H}^{*}(x,u_{n}(x))\dx\leq \int_{\R^N}\mathcal{H}^{*}(x,u_{n}(x))\dx<c_{8}$$
since $\{u_{n}\}$ is bounded in $W$, for some $c_{7},c_{8}>0$.

Therefore, we have $\sup\limits_{n\in\N}|A_n|<\infty$. Using $(K_{1})$, we get
 $$\lim_{d\rightarrow \infty}\int_{A_{n}\cap B_{d}(0)^{c}}|K(x)|^{\frac{2N}{2N-\lambda}}\dx=0 \ \text{uniformly in } n \in\mathbb{N}$$ consequently, for a given $\epsilon>0$ there exists $d_0>0$ such that
\begin{equation}\label{l5.3}
\int_{A_{n}\cap B_{d_0}(0)^{c}}|K(x)|^{\frac{2N}{2N-\lambda}}\dx< \epsilon^{\frac{b}{(b-1)}}\ \text{ for each } n.
\end{equation}
Using H$\ddot{\text{o}}$lder's inequality and  Proposition \ref{prop9}, we have 
 \begin{align*}
\int_{B_{d_0}(0)^{c}}&Q(x)\left| (\mathcal{H}^{*}(x,u_{n}(x)))^{\frac{b-1}{h^{*}_{i}}}\chi_{[t_{0},t_{1}]}(u_{n}(x))(u_{n}(x)-u(x))\right|^{l}\dx\\
&\leq c_{9}\int_{A_{n}\cap B_{d_0}(0)^{c}}Q(x)|u_{n}(x)|^{(b-1)l}\chi_{[t_{0},t_{1}]}(|u_{n}(x)|) |u_{n}(x)-u(x)|^{l}\dx\\
&\leq c_{9}\max_{i\in\{1,2\}}\left\lbrace\left( \int_{A_{n}\cap B_{d_0}(0)^{c}}Q(x)|u_{n}(x)|^{bl}\chi_{[t_{0},t_{1}]}(u_{n}(x))\dx\right)^{\frac{(b-1)}{b}}\right.  \\
& \left. \left( \int_{A_{n}\cap B_{d_0}(0)^{c}}Q(x)|u_{n}(x)-u(x)|^{bl}\dx\right)^{\frac{1}{b}} \right\rbrace\\
&\leq c_{10} t_1^{(b-1)l}\left(\int_{A_{n}\cap B_{d_0}(0)^{c}}Q(x)\dx\right)^{\frac{(b-1)}{b}}
\end{align*}
for some $c_{9},c_{10}>0$.

Further, by \eqref{l5.3}, we obtain
\begin{equation}\label{ps3}
\int_{B_{d_0}(0)^{c}}Q(x)\left| (\mathcal{H}^{*}(x,u_{n}(x)))^{\frac{b-1}{h^{*}_{i}}}\chi_{[t_{0},t_{1}]}(u_{n}(x))(u_{n}(x)-u(x))\right|^{l}\dx\leq c_{11} \epsilon,
\end{equation} 
for some $c_{11}>0$.

By  Propositions   \ref{prop7}, \ref{prop9},  Lemma \ref{compact}, Theorem \ref{embed} and H$\ddot{\text{o}}$lder's inequality,  we have
\begin{align*}
&\int_{B_{d_0}(0)^{c}}Q(x)\left| \left\lbrace \psi_x(x,u_{n}(x))u_{n}(x)+  (\mathcal{H}^{*}(x,u_{n}(x)))^{\frac{b-1}{h^{*}_{i}}}\right\rbrace (u_{n}(x)-u(x))\right| ^{l}\dx\\
&\leq c_{12} \max_{i\in\{1,2\}}\left\lbrace \int_{\R^{N}}Q(x)\left( |u_{n}(x)|^{(\psi_{i}-1)l}+|u_{n}(x)|^{(b-1)l}\right) |u_{n}(x)-u(x)|^{l}\dx\right\rbrace \\
&\leq  c_{12}\max_{i\in\{1,2\}}\left\lbrace\left( \int_{\R^{N}}Q(x)|u_{n}(x)|^{\psi_{i}l}\dx\right)^{\frac{(\psi_{i}-1)}{\psi_{i}}} \right.\\
&\hspace{3 cm}\left.\left( \int_{\R^{N}}Q(x)|u_{n}(x)-u(x)|^{\psi_{i}l}\dx\right)^{\frac{1}{\psi_{i}}} \right\rbrace\\
&+  c_{12}\max_{i\in\{1,2\}}\left\lbrace\left( \int_{\R^{N}}Q(x)|u_{n}(x)|^{bl}\dx\right)^{\frac{(b-1)}{b}}\right. \\
&\hspace{3 cm}\left.\left( \int_{\R^{N}}Q(x)|u_{n}(x)-u(x)|^{bl}\dx\right)^{\frac{1}{b}} \right\rbrace\\
&\leq c_{13} \max_{i\in\{1,2\}}\left\lbrace  \|u_{n}(x)-u(x)\|_{L^{\psi_{i}l}_{Q}(\R^{N})}^{l} +\|u_{n}(x)-u(x)\|_{L^{bl}_{Q}(\R^{N})}^{l} \right\rbrace
\end{align*}
 for some $c_{12},c_{13}>0$.

Therefore, by Theorem \ref{embed}, we obtain
\begin{align}\label{ps4}
\int_{B_{d_0}(0)^{c}}Q(x)&\left| \left\lbrace \psi_x(x,u_{n}(x))u_{n}(x)+  (\mathcal{H}^{*}(x,u_{n}(x)))^{\frac{b-1}{h^{*}_{i}}}\right\rbrace (u_{n}(x)-u(x))\right| ^{l}\dx\notag \\
&\leq c_{14},
\end{align} 
for some $c_{14}>0$.

Consequently, \eqref{ps2}, \eqref{ps3} and \eqref{ps4} implies
$$\int_{B_{d_0}(0)^{c}} |K(x)f(u_{n}(x))(u_{n}(x)-u(x))|^{l}\dx\\
\leq  2^{l-1}\epsilon^{l} c_{14}
+2^{l-1}c_{\epsilon}^{l}c_{11} \epsilon\rightarrow 0$$
as $\epsilon$ was arbitrary.

On the other side, by $(f_{3})$ and Strauss compactness lemma \cite[Theorem A.I]{strauss}, we have
$$\lim_{n\rightarrow \infty}\int_{B_{d_0}(0)} |K(x)f(u_{n}(x))(u_{n}(x)-u(x))|^{l}\dx=0,$$
which completes the proof.
\end{proof}

\textbf{Proof of the Theorem \ref{t1}}. Both the geometric conditions of the mountain pass theorem follow from Lemmas \ref{l1} and \ref{l2}. Next, we will prove that the functional $I$ satisfies the $(PS)_{c_{M}}$ condition.

Let $\{u_{n}\}\subseteq W$ be any Palais-Smale sequence, i.e., $I(u_{n})\rightarrow c_{M}$ and $I'(u_{n}) \rightarrow 0$ in dual space of $W$. 
By Lemma \ref{l3}, we conclude that $\{u_{n}\}$ is bounded in $W$ and
$u_{n}\rightharpoonup u$ weakly in $W$.
As a consequence, $I'(u_{n})(u_{n}-u)\rightarrow 0$ as $n\rightarrow\infty$, i.e.,
\begin{equation*}
J'(u_{n})(u_{n}-u) -\int_{\R^{N}}\int_{\R^{N}}\dfrac{K(x) K(y) F(u_{n}(x))f(u_{n}(y))(u_{n}(y)-u(y))}{|x-y|^\lambda}\dx \dy \rightarrow 0 
\end{equation*}
$\hbox{as} \  n\rightarrow 0.$
Next, we claim that that 
\begin{equation}\label{ps1}
\int_{\R^{N}}\int_{\R^{N}}\dfrac{K(x) K(y) F(u_{n}(x))f(u_{n}(y))(u_{n}(y)-u(y))}{|x-y|^\lambda}\dx \dy \rightarrow 0 \ \hbox{as} \  n\rightarrow 0.\end{equation}
Let $|K(y)|^{l}=Q(y)$. By \eqref{l1.1}, we have 
\begin{equation}\label{l5.1}
\|K(x)F(u_{n}(x))\|_{L^{l}(\R^{N})} \leq c_{15},
\end{equation} 
for some $c_{15}>0$ since $\{u_{n}\}$ is bounded in $W$. By Lemma \ref{l5}, we have 
\begin{equation}\label{l5.2}
\lim_{n\rightarrow\infty}\int_{\R^{N}}|K(x)f(u_{n}(x))(u_{n}(x)-u(x))|^{l}\dx = 0.
\end{equation}
By \eqref{l5.1}, \eqref{l5.2} and Proposition \ref{hardy},  the claim in the \eqref{ps1} is proved. 
 Hence, by Lemma \ref{ps}, we have $u_{n}\rightarrow u$.
Thus, $(PS)_{c_{M}}$ condition is satisfied for the functional $I$.
 
 Hence, by the  mountain pass theorem, there exists a critical point $u_{M}$ of $I$ with level $c_{M}$, i.e., $I'(u_{M})=0$ and $I(u_{M})=c_{M}>0$. Thus, $u_{M}$ is the non-trivial weak solution of the Problem \eqref{1.1}. \qed
\section{Ground State Solution}\label{sec6}
\setcounter{section}{6} \setcounter{equation}{0}

In this section, we prove that the solution obtained through Theorem \ref{t1} is a ground state solution. Let us recall the definition of a ground state solution:
 \begin{definition}
 A weak solution $u_{0}$ of \eqref{1.1} is called a ground state solution if it has the least energy, i.e.,
 we say the solution $u_{0}$ is a ground state solution of \eqref{1.1} if 
\begin{equation}\label{n1}
 I(u_{0})=r=\displaystyle\inf_{u\in S}I(u), 
 \end{equation} where $S$ is the set of all critical points of the functional $I$.
 \end{definition}
 
 To prove that the solution obtained in Theorem \ref{t1} is a ground state solution, we use the minimization  method, in particular, the Nehari manifold method. We define  
 \begin{equation}\label{n2}
 \aleph=\{u\in W\backslash\{0\}| I'(u)u=0\} \text{ and } b=\displaystyle\inf_{u\in\aleph}I(u). 
 \end{equation}
  The set $\aleph$ is called the Nehari manifold.  It can be observed that  $S\subseteq \aleph$. The key idea of this method is to search for a non-trivial  critical point of $I$ in $\aleph$ instead of the whole space $W$.  
 To know more about this method, one can refer to \cite{nehari}.
 The existence of a ground state  solution is proved by many researchers; we refer to \cite{bueno,chen,luo,missaoui,moroz2,moroz3} and references therein.

For $u\in W$, define the function, $h_{u}:[0,\infty)\rightarrow\R$ such that $h_{u}(t)=I(tu)$, i.e.,
\begin{align*}
&h_{u}(t)=\int_{\R^{N}}\int_{\R^{N}}\mathcal{H}\left( x,y,\dfrac{|tu(x)-tu(y)|}{|x-y|^{s}}\right)\dfrac{\dx\dy}{|x-y|^{N}}\\
&+ \int_{\R^{N}}V(x)\mathcal{H}_{x}(x,|tu(x)|)\dx -\frac{1}{2}\int_{\R^{N}}\int_{\R^{N}}\dfrac{K(x) K(y) F(tu(x))F(tu(y))}{|x-y|^\lambda}\dx \dy.
\end{align*}
\begin{lemma}\label{gs2}
Let  $(f_{1})-(f_{4})$, $(GS)$, $({K}_{1})-({K}_{2})$ and $(\mathcal{H}_{1})-(\mathcal{H}_{4})$ hold. If $u\in W\backslash\{0\},$ then there exists a unique $t_{u}>0$ such that $t_{u}u\in\aleph$. Moreover, $\max\limits_{t\in[0,\infty]}h_{u}(t)=h_{u}(t_{u})=I(ut_{u})$.
 \end{lemma}
\begin{proof}

 We observe that $h'_{u}(t)=0$ if and only if $tu\in\aleph$.  Lemma \ref{l1} and Lemma \ref{l2}  imply that $h_{u}(t)>0 \ \text{for sufficiently small}  \ t$ and $h_{u}(t)<0 \ \text{for sufficiently large} \ t.$ Thus $\exists \ t_{u}\in(0,\infty)$ such that $\max\limits_{t\in[0,\infty]}h_{u}(t)=h_{u}(t_{u})=I(ut_{u})$. Consequently, $h'_{u}(t_{u})=0$ and $t_{u}u\in\aleph$. Next, we will prove the uniqueness of $t_{u}$. If $t$ is the critical point of $h_{u},$ then we have
 \begin{align*}
  h'_{u}(t)=&\int_{\R^{N}}\int_{\R^{N}} h\left(x,y,\dfrac{|tu(x)-tu(y)|}{|x-y|^{s}}\right)\dfrac{(tu(x)-tu(y))^{2}}{t|x-y|^{N+2s}}\dx\dy\\
 &+ \int_{\R^{N}}\frac{V(x)h_{x}(x,|tu(x)|)(tu(x))^{2}\dx}{t}-\\
 &\int_{\R^{N}}\int_{\R^{N}}\dfrac{K(x) K(y) F(u(x))f(tu(y))u(y)}{|x-y|^\lambda}\dx \dy=0,
 \end{align*}
which implies that
\begin{align}\label{gs1}
 \notag \int_{\R^{N}}\int_{\R^{N}} & h\left(x,y,\dfrac{|tu(x)-tu(y)|}{|x-y|^{s}}\right)\dfrac{(tu(x)-tu(y))^{2}}{t^{h_{2}}|x-y|^{N+2s}}\dx\dy\notag\\
 &+\int_{\R^{N}}\frac{V(x)h_{x}(x,|tu(x)|)(tu(x))^{2}\dx}{t^{h_{2}}}\notag\\
&= \int_{\R^{N}}\int_{\R^{N}}\dfrac{K(x) K(y) F(tu(x))f(tu(y))tu(y)}{t^{h_{2}}|x-y|^\lambda}\dx \dy.
\end{align}
Continuing as in \cite[Lemma 4.3]{silva2}, one can check that the right-hand side of \eqref{gs1} is decreasing for $t>0$. Consider,
\begin{align*}
 \int_{\R^{N}}\int_{\R^{N}}&\dfrac{K(x) K(y) F(tu(x))f(tu(y))tu(y)}{t^{h_{2}}|x-y|^\lambda}\dx \dy\\
&= \int_{\R^{N}}K(y)\left( \int_{\R^{N}}\dfrac{ K(x) F(tu(x))\dx}{|x-y|^\lambda}\right) \frac{f(tu(y))tu(y)}{t^{h_{2}}}\dy\\
&= \int_{\R^{N}}K(y)\left( \int_{\R^{N}}\dfrac{ K(x) F(tu(x))\dx}{t^\frac{h_{2}}{2}|x-y|^\lambda}\right) \frac{f(tu(y))|u(y)|^{\frac{h_{2}}{2}}}{|tu(y)|^{\frac{h_{2}}{2}-2}(tu(y))}\dy
\end{align*}
which implies the left-hand side of \eqref{gs1} is increasing strictly for $t>0$ by $(GS)$ and $(f_{4})$. Therefore, $t_{u}$ is a unique  critical point of $h_{u}$. 
\end{proof}
\textbf{Proof of Theorem \ref{t2}}.
It is enough to prove that $c_{M}=b=r$, where $b$ and $r$ are defined in \eqref{n1} and \eqref{n2}, respectively.

By using the fact that $S\subseteq \aleph$, we have $b\leq r.$ Also, it can be seen  $ r\leq c_{M}$. It will be sufficient to prove that $b\geq  c_{M}$. 

If $v\in\aleph,$ then $h'_{v}(1)=0$. By Lemma \ref{gs2}, we have $\max\limits_{t\in[0,\infty]}h_{v}(t)=h_{v}(1)=I(v)$.
 
 Choose a function $\gamma:[0,1]\rightarrow  W$ such that $\gamma(t)=tt_{0}v,$ where $t_{0}>0$ such that $I(t_{0}v)<0$, which implies that $\gamma\in\varGamma$. 
  Therefore, we have $c_{M}\leq\max\limits_{t\in[0,1]}I(\gamma(t))=\max\limits_{t\in[0,1]}I(tt_{0}v)\leq\max\limits_{t\geq 0}I(tv)=I(v),$ which is true for every element  $v\in\aleph$. Hence,  $b\geq  c_{M},$ which completes the proof.
\qed

 \section*{\small
 Conflict of interest} 
 {\small
 The authors declare that they have no conflict of interest.}



\end{document}